\documentclass[12pt]{amsart}
\usepackage{amscd,amssymb,times, amsmath, fullpage}
\usepackage{enumitem}
\newtheorem{theorem}{Theorem}
\newtheorem{lemma}[theorem]{Lemma}

\newtheorem{corollary}[theorem]{Corollary}
\theoremstyle{definition}

\theoremstyle{remark}
\newtheorem{remark}[theorem]{Remark}

\newcommand{\FF}{\mathcal{ F}}

\newcommand\numberthis{\addtocounter{equation}{1}\tag{\theequation}}

\def\bS{{\mathbb S}}

\def\a{\alpha}

\catcode`\@=\active 

\author{Gianluca Bande}
\address{Dipartimento di Matematica e Informatica, Universit{\`a} degli studi di Cagliari, Via Ospedale 72, 09124 Cagliari, Italy}
\email{gbande{\char'100}unica.it}
\author{David E. Blair} 
\address{Department of Mathematics, Michigan State University, East Lansing, MI 48824--1027, USA}
\email{blair{\char'100}math.msu.edu}
\author{Amine Hadjar}
\address{Laboratoire de Math{\'e}matiques, Informatique et
Applications, Universit{\'e} de Haute Alsace - 4, Rue des
Fr{\`e}res Lumi{\`e}re, 68093 Mulhouse C\'edex, France}
\email{mohamed.hadjar{\char'100}uha.fr}

\thanks{The first author is supported by P.R.I.N. 2010/11 -- Variet\`{a}  reali e complesse: geometria, topologia e analisi armonica -- Italy.}

\begin{document}

\title{Bochner and Conformal Flatness of Normal Metric Contact Pairs}

\date{\today; MSC 2010 classification: primary 53C55; secondary 53C25, 53D10, 53A30, 53C15, 53D15}
\maketitle

\begin{abstract}
We prove that the normal metric contact pairs with orthogonal characteristic foliations, which are either Bochner-flat or locally conformally flat, are locally isometric to the Hopf manifolds. As a corollary we obtain the classification of locally conformally flat and Bochner-flat non-K\"ahler Vaisman manifolds.
\end{abstract}

\section{Introduction}

In 1949 S. Bochner \cite{Boch} introduced in K\"ahler geometry a new tensor field as a formal analogue of the Weyl conformal curvature tensor.  The Bochner tensor was extended from the K\"ahler manifold setting to general almost Hermitian manifolds by Tricerri and Vanhecke \cite{TV}.  While the Bochner tensor and Bochner-flatness have been studied  in 
K\"ahler geometry by a number of authors over the years, there have not been many applications to Hermitian manifolds that are not always K\"ahler.  In \cite{BM} V. Mart\'\i n-Molina and the second author showed that there are no conformally flat normal complex contact metric manifolds but that a Bochner-flat normal complex contact metric manifold must be
K\"ahler and locally isometric to ${\mathbb C}P^{2n+1}(4)$. It should be remarked that the notion of normality used is that due to Korkmaz \cite{K} and includes such non-K\"ahler normal complex contact metric manifolds as the complex Heisenberg group.

  A normal metric contact pair (with decomposable $\phi$) 
carries two complex structures and hence two Bochner tensors.  In the present paper we show that if either  Bochner tensor vanishes the manifold is locally isometric to the Hopf manifold $\bS^{2m+1}(1)\times \bS^1$.  As a corollary we will see that a conformally flat normal metric contact pair must also be locally isometric to a Hopf manifold. Moreover, as a bypass result, we recover the classification of locally conformally flat and Bochner-flat non-K\"ahler Vaisman manifolds.

\section{Preliminaries}

Contact pairs were introduced by G. D.  Ludden,  K. Yano and the second author in \cite{BLY} under the name
{\it bicontact} and by the first and third authors in \cite{B,BH} with the name {\it contact pair}.
A pair of 1-forms $(\a_1,\a_2)$ on a manifold $M$ is said to be a {\it contact pair of type $(m,n)$} if
\vskip4pt
\centerline{$\a_1\wedge(d\a_1)^m\wedge\a_2\wedge(d\a_2)^n$ is a volume form,}
\vskip4pt
\centerline{$(d\a_1)^{m+1}=0$ and $(d\a_2)^{n+1}=0$.}
\vskip4pt
\noindent
While it is possible to consider a contact pair of type (0,0), it seems most natural to require at least one of the forms to resemble a contact form, so that at least one of  $m$ or $n$ will be positive.

We can naturally associate to a contact pair two subbundles of the tangent bundle $TM$:
$$\{ X: \a_i(X)=0, d\a_i(X,Y)=0\; \forall \; Y\},\; i=1,2$$
These subbundles are integrable \cite{B, BH} and determine the {\it characteristic foliations} of $M$, denoted 
${\FF}_1$ and ${\FF}_2$ respectively.  The characteristic foliations are transverse and complementary and the leaves of ${\FF}_1$ and ${\FF}_2$ are contact manifolds of dimension $2n+1$ and $2m+1$ respectively, with contact forms induced by $\a_2$ and $\a_1$.  We also define the $(2m+2n)$-dimensional {\it horizontal subbundle} $\mathcal{H}$ to be the intersection of the kernels of $\a_1$ and $\a_2$.

The equations
\begin{eqnarray*}
&\alpha_1 (Z_1)=\alpha_2 (Z_2)=1  , \; \; \alpha_1 (Z_2)=\alpha_2
(Z_1)=0 , \\
&i_{Z_1} d\alpha_1 =i_{Z_1} d\alpha_2 =i_{Z_2}d\alpha_1=i_{Z_2}
d\alpha_2=0
\end{eqnarray*}
where $i_X$ is the contraction with the vector field $X$,  determine uniquely the two vector fields $Z_1$ and $Z_2$, called \emph{Reeb vector fields}. Since they commute, they give rise to a locally free $\mathbb{R}^2$-action, called  the \emph{Reeb action}. 

A {\it contact pair structure} \cite{BH1} on a manifold $M$ is a triple
$(\alpha_1 , \alpha_2 , \phi)$, where $(\alpha_1 , \alpha_2)$ is a
contact pair and $\phi$ a tensor field of type $(1,1)$ such that:
$$\phi^2=-Id + \alpha_1 \otimes Z_1 + \alpha_2 \otimes Z_2 , \quad
\phi Z_1=\phi Z_2=0.$$
The rank of $\phi$ is $\dim M-2$ and
 $\alpha_i \circ \phi =0$ for $i=1,2$. 
 
 The endomorphism $\phi$ is said to be {\it decomposable} if
$\phi (T\mathcal{F}_i) \subset T\mathcal{F}_i$, for $i=1,2$.
When $\phi$ is decomposable, $(\alpha_1 , Z_1 ,\phi)$ (respectively
$(\alpha_2 , Z_2 ,\phi)$) induces, on every leaf of $\mathcal{F}_2$ (respectively $\mathcal{F}_1$), a contact form and the restriction  $\phi_i$ of $\phi$ to the leaf forms an almost contact structure 
$(\alpha_i, Z_i,\phi_i)$.  It is important to note that there exists contact pair structures with decomposable 
$\phi$ which are not locally products \cite{BH2}.

In \cite{BH2} the first and third authors introduced two natural almost complex structures on the manifold $M$
by
$$J=\phi -\a_2\otimes Z_1+\a_1\otimes Z_2,\quad T=\phi +\a_2\otimes Z_1-\a_1\otimes Z_2.$$
The contact pair structure is said to be {\it normal} if both of these almost complex structures are integrable.  

On manifolds endowed with contact pair structures it is natural
to consider the following metrics \cite{BH1}.
Let $(\alpha_1 , \alpha_2 ,\phi )$ be a contact pair structure on
 $M$. A Riemannian metric $g$ on $M$ is said to be \emph{associated} if 
$$g(X, \phi Y)= (d \alpha_1 + d\alpha_2) (X,Y),$$
$$g(X, Z_i)=\alpha_i(X),\;i=1,2.$$ 
A {\it metric contact pair} on a manifold $M$ is a
four-tuple $(\alpha_1, \alpha_2, \phi, g)$ where $(\alpha_1,
\alpha_2, \phi)$ is a contact pair structure and $g$ an associated
metric with respect to it. Such a manifold will also be called a metric contact pair.
We note that on a normal metric contact pair, the Reeb vector fields are Killing.

For a metric contact pair,  $(\alpha_1, \alpha_2, \phi, g)$, the endomorphism field $\phi$ is decomposable
if and only if the characteristic foliations $\mathcal{F}_1 ,
\mathcal{F}_2$ are orthogonal. In this case $(\alpha_i, \phi, g)$ induces a contact metric structure $(\alpha_i, \phi_i, g)$ on the leaves of $\mathcal{F}_j$ , for $j\neq i$ .
Thus we assume the decomposability of $\phi$ throughout.
By the normality each $(\alpha_i, \phi_i, g)$ is a Sasakian structure on each leaf of the characteristic foliations.
Moreover the leaves are  minimal submanifolds \cite{BH3}.

\vspace{0.5cm}
In the course of our work we will need the following lemmas. Some formulas from \cite{BB, BH1} can be summarized as follows:
\begin{lemma}\label{Lemma-2.2}
On a normal metric contact pair with decomposable $\phi$, for every $X$
we have
$$\nabla_XZ_1=-\phi_1X,\quad \nabla_XZ_2=-\phi_2X.$$
\end{lemma}
Using the previous lemma, we can restate Corollary $3.2$ of \cite{BBH} as follows:
\begin{lemma}\label {Lemma-2.1} On a normal metric contact pair with decomposable $\phi$, the covariant derivative of $\phi$ is given by
$$g((\nabla_X \phi)Y, V)=  \sum_{i=1} ^2 \bigl (d\alpha_i (\phi Y , X) \alpha_i (V)
- d\alpha_i (\phi V,X) \alpha_i (Y)\bigr )$$
and hence for the almost complex structure $J$ we have
\begin{align*}
g((\nabla_XJ)Y,V)=&\sum_{i=1} ^2 \bigl (d\alpha_i (\phi Y , X) \alpha_i (V)- d\alpha_i (\phi V,X) \alpha_i (Y)\bigr )\\
&-d\alpha_2(X,Y)\a_1(V)-d\alpha_1(X,V)\a_2(Y)\\
&+d\alpha_1(X,Y)\a_2(V)+d\alpha_2(X,V)\a_1(Y).
\end{align*}

\end{lemma}

Our conventions for the curvature tensor of a Riemannian manifold are
\begin{align*}
R(X,Y)V&=\nabla_X\nabla_YV-\nabla_Y\nabla_XV-\nabla_{[X,Y]}V,\\
R(X,Y,V,W)&=g(R(X,Y)V,W).
\end{align*}

Let $Z=Z_1+Z_2$; from Theorem $2$ of \cite{BH4} or the proof of Lemma $2.1$ of \cite{BB} one readily has the following lemma.

\begin{lemma}\label{Lemma-2.3} On a normal metric contact pair with decomposable $\phi$, 
\begin{align*}
g(R_{XY}Z,V)=&d\a_1(\phi V,X)\a_1(Y) +d\a_2(\phi V,X)\a_2(Y)\\
&-d\a_1(\phi V,Y)\a_1(X)-d\a_2(\phi V,Y)\a_2(X).
\end{align*}

\end{lemma}

\begin{lemma}\label{Lemma-2.4}  On a normal metric contact pair with decomposable $\phi$ let $X$ be a local unit vector field in $T{\FF}_2\cap{\mathcal{H}}$. Then
$$R(X,Z_1,Z_1,X)=1,\quad R(X,Z_1,Z_2,X)=0,\quad R(X,Z_2,Z_2,X)=0.$$
\end{lemma}

\begin{proof} Using Lemmas \ref{Lemma-2.2} and \ref{Lemma-2.1}, we have
\begin{align*}
R(X,Z_1,Z_1,X)=&g(-\nabla_{Z_1}\nabla_XZ_1-\nabla_{[X,Z_1]}Z_1,X)\\
=&g(\nabla_{Z_1}\phi_1X+\phi_1[X,Z_1],X)\\
=&g((\nabla_{Z_1}\phi_1)X+\phi_1(\nabla_XZ_1),X)\\
=&g(-\phi_1^2X,X)\\
=&1.
\end{align*}
Similarly
$$R(X,Z_1,Z_2,X)=g(-\nabla_{Z_1}\nabla_XZ_2-\nabla_{[X,Z_1]}Z_2,X)=0.$$
Finally, since $Z_2$ is Killing,
\begin{align*}
R(X,Z_2,Z_2,X)=&g(-\nabla_{[X,Z_2]}Z_2,X)\\=&g(\nabla_{\nabla_{Z_2}X}Z_2,X)\\
=&-g(\nabla_XZ_2,\nabla_{Z_2}X)\\
=&0.
\end{align*}
\end{proof}
\vskip6pt

We denote by $\rho(R)$ or $\rho(X,Y)$ or simply $\rho$ the Ricci tensor of $R$ and by $\tau$ the scalar curvature. For an almost Hermitian manifold with almost complex structure $J$,  we recall the
 {\it $*$-Ricci tensor} defined by $\rho^*(X,Y)=\sum_i R(X,e_i,Je_i,JY)$ where  $\{e_i\}$ is an arbitrary 
 orthonormal basis.  In general $\rho^*$ is not symmetric but it does satisfy 
 $\rho^*(X,Y)=\rho^*(JY,JX)$. The trace of $\rho^*$, denoted  $\tau^*$,  is called the {\it $*$-scalar curvature}.
\vskip6pt
A normal metric contact pair carries two Hermitian structures $(g, J)$ and $(g, T)$, with respect to the same metric $g$, giving rise to a pair of $*$-Ricci tensors. In the sequel, $\rho^*$ will denote the $*$-Ricci tensor of $(g, J)$.

\begin{lemma}\label{Lemma-2.5} 
On a normal metric contact pair with decomposable $\phi$,
\begin{align*}
\rho^*(X,Y)=&\rho(X,Y)-(2m-1)g(\phi_1X,\phi_1Y)-(2n-1)g(\phi_2X,\phi_2Y)\\
&-2m\a_1(X)\a_1(Y)-2n\a_2(X)\a_2(Y).
\end{align*}
Moreover $\rho^*$ is symmetric and $J$-invariant and we have
$$\tau-\tau^*=4(m^2+n^2).$$
\end{lemma}
\begin{proof}
On an almost Hermitian manifold we have the following relation between the Ricci tensor and $*$-Ricci tensor, \cite[p. $195$]{Y}
$$(\rho_{jt}-\rho^*_{jt})J^t{}_i=\nabla_t\nabla_jJ^t{}_i-\nabla_j\nabla_tJ^t{}_i.$$
From Lemmas \ref{Lemma-2.2} and \ref{Lemma-2.1} we have
$$\nabla_tJ^t{}_i=-2m(\a_1)_i-2n(\a_2)_i$$
and differentiating
$$\nabla_j\nabla_tJ^t{}_i=-2m(\phi_1)_{ji}-2n(\phi_2)_{ji}.$$
In like manner, differentiation of the formula of Lemma \ref{Lemma-2.1} also yields
$$\nabla_t\nabla_jJ^t{}_i=-\phi_{ji}-2m(\a_2)_i(\a_1)_j+2n(\a_1)_i(\a_2)_j.$$
Using these we have
$$(\rho_{jt}-\rho^*_{jt})J^t{}_i=
(2m-1)(\phi_1)_{ji}+(2n-1)(\phi_2)_{ji}-2m(\a_2)_i(\a_1)_j+2n(\a_1)_i(\a_2)_j.$$
Or in an invariant language
\begin{align*}
 \rho^*(Y,JX)-\rho(Y,JX)=&-(2m-1)d\a_1(Y,X)-(2n-1)d\a_2(Y,X)\\
&+2m\a_1(Y)\a_2(X)-2n\a_1(X)\a_2(Y).
\end{align*}
Replacing $X$ by $JX$ we have
\begin{align*}
\rho(Y,X)-\rho^*(Y,X)=&(2m-1)g(\phi_1Y,\phi_1X)+(2n-1)g(\phi_2Y,\phi_2X)\\
&+2m\a_1(X)\a_1(Y)+2n\a_2(X)\a_2(Y)
\end{align*}
as desired. By the previous relation and the symmetry of $\rho$ we get the symmetry of $\rho^*$. Next, by the fact that $\rho^*(X,Y)=\rho^*(JY,JX)$, we obtain the $J$-invariance of $\rho^*$.
Contracting in the same relation, we have for the scalar curvature
$\tau-\tau^*=4(m^2+n^2)$. 
\end{proof}

\vskip6pt
Now consider a local orthonormal basis $\{E_1,\dots,E_m,E_{m+1},\dots,E_{m+n},Z_1,Z_2\}$ where the first $m$ vector fields are tangent to the leaves of ${\FF}_2$ and the next $n$ tangent to the leaves of 
${\FF}_1$.
From Lemma \ref{Lemma-2.3} we have the following.
$$\rho(Z_1,Z_1+Z_2)=\sum_{i=1}^m d\a_1(\phi E_i,E_i)=2m.$$
Similarly $\rho(Z_2,Z_1+Z_2)=2n$.  For simplicity we abbreviate $\rho(Z_i,Z_j)$ by $\rho_{ij}$ and the same for $\rho^*$.  We then have the following lemma.

\begin{lemma}\label{Lemma-2.6}
On a normal metric contact pair with decomposable $\phi$, $\rho$ is $J$-invariant on horizontal vectors,
i.e. $\rho(JX,JY)=\rho(X,Y)$, and
$$\rho_{11}=2m,\quad\rho_{22}=2n,\quad\rho_{12}=0,\quad
\rho^*_{11}=\rho^*_{22}=\rho^*_{12}=0.$$
\end{lemma}

\begin{proof}
From the symmetry of $\rho^*$ and the formula of Lemma \ref{Lemma-2.5}, the $J$-invariance of $\rho$ restricted to $\mathcal{H}$ is immediate. By the $J$-invariance of $\rho^*$ we have $\rho^*(Z,Z)=\rho^*(Z_2-Z_1,Z_2-Z_1)$ giving $\rho^*_{12}=0$.
The formula of Lemma \ref{Lemma-2.5} then gives $\rho_{12}=0$. We noted above that $\rho(Z_1,Z)=2m$ and 
$\rho(Z_2,Z)=2n$ and therefore $\rho_{11}=2m$ and $\rho_{22}=2n$. Finally Lemma \ref{Lemma-2.5} gives
$\rho^*_{11}=\rho^*_{22}=0$.
\end{proof}

\section{The Bochner Curvature Tensor}

In \cite{TV} F. Tricerri and L. Vanhecke gave a complete decomposition of the space of curvature tensors over a Hermitian vector space into irreducible factors under the action of the unitary group similar to the well known decompostion of curvature tensors in the Riemannian setting with the action of the orthogonal group.  In the real case the Weyl conformal curvature tensor emerges in a natural manner and correspondingly the Bochner tensor for an almost Hermitian manifold emerges as one factor of the decomposition.  The Bochner tensor in almost Hermitian geometry is considerably more complicated than that in K\"ahler geometry and as a result has not been studied extensively.

Define   $(0,4)$-tensors $\pi_1$, $\pi_2$ and $L_3 R$ by:
\begin{align*}
\pi_1(X,Y,Z,W)&=g(X,Z)g(Y,W)-g(Y,Z)g(X,W),\\
\pi_2(X,Y,Z,W)&=2g(JX,Y)g(JZ,W)+g(JX,Z)g(JY,W)-g(JY,Z)g(JX,W),\\
L_3 R(X,Y,Z,W)&=R(JX,JY,JZ,JW).
\end{align*}

Given a $(0,2)$-tensor $S$, we denote by $\varphi(S)$ and $\psi(S)$:
\begin{align*}
\varphi(S)(X,Y,Z,W)=&g(X,Z)S(Y,W)+g(Y,W)S(X,Z)\\
&-g(X,W)S(Y,Z)-g(Y,Z)S(X,W),\\
\psi(S)(X,Y,Z,W)=&2g(X,JY)S(Z,JW)+2g(Z,JW)S(X,JY)\\
&+g(X,JZ)S(Y,JW)+g(Y,JW)S(X,JZ)\\
&-g(X,JW)S(Y,JZ)-g(Y,JZ)S(X,JW).
\end{align*}

Taking into account our sign convention for the curvature tensor and complex dimension of our complex manifold, the normal metric contact pair $M$, the Bochner tensor of Tricerri and Vanhecke is given by the following.
Given a metric contact pair of complex dimension $m+n+1>2$, the {\it Bochner tensor} corresponding to the complex structure $J$ is defined as
\begin{align*}
B=&R+\frac{1}{4(m+n+2)}\psi(\rho^*)(R-L_3R)+\frac{1}{4(m+n)}\varphi(\rho)(R-L_3R)\\
&+\frac{1}{16(m+n+3)} (\varphi+\psi)(\rho+3\rho^*)(R+L_3R)\\
&+\frac{1}{16(m+n-1)} (3\varphi-\psi)(\rho-\rho^*)(R+L_3R)\\
&-\frac{\tau+3\tau^*}{16(m+n+2)(m+n+3)} (\pi_1+\pi_2)
-\frac{\tau-\tau^*}{16(m+n-1)(m+n)}(3\pi_1-\pi_2).
\end{align*}
We will work almost exclusively with this complex structure.  When discussing both $J$ and $T$ we will denote the corresponding Bochner tensors by $B_J$ and $B_T$ respectively.

When the complex dimension is $2$, the formula for the Bochner tensor is somewhat different and the formula given in \cite{TV} is in error; the correct formula is the following.
\begin{align*}
B=&R+\frac{1}{12}\psi(\rho^*)(R-L_3R)+\frac{1}{4}\varphi(\rho)(R-L_3R)\\
&+\frac{1}{64}(\varphi+\psi)(\rho+3\rho^*)(R+L_3R)
-\frac{\tau+3\tau^*}{192} (\pi_1+\pi_2)
+\frac{\tau-\tau^*}{32}(3\pi_1-\pi_2).
\end{align*}

\section{Main Results}

We now turn to our main result on Bochner-flatness and then discuss the question of conformal flatness as a corollary.

\begin{theorem}\label{Theorem-1}
Let $(M,\a_1,\a_2,\phi,g)$ be a normal metric contact pair with decomposable $\phi$.  If either of the two Bochner tensors, $B_J$ or $B_T$,  vanishes, then $M$ is locally isometric to the Hopf manifold $\bS^{2m+1}(1)\times \bS^1$.
\end{theorem}
The proof will be given in three stages.  We will first prove the theorem for the Bochner tensor $B_J$ to be denoted simply by $B$ for complex dimension $m+n+1>2$ (Stage 1).  Then we will prove the theorem for the case of complex dimension 2 (Stage 2).  Finally we indicate the proof for the Bochner tensor $B_T$ (Stage 3).

\begin{proof}
Stage 1, Step 1:\\
We begin by evaluating $B$ on the Reeb vector fields, in particular we have $0=B(Z_1,Z_2,Z_2,Z_1)$.  We proceed term by term in the definition of $B$. That $R(Z_1,Z_2,Z_2,Z_1)=0$ is immediate since $\nabla_{Z_i}Z_j=0$.  The next term in the definition vanishes by virtue of $\rho^*$ being $J$-invariant (Lemma \ref{Lemma-2.5}).  In the third term the Ricci tensors in  $\varphi(\rho)(R-L_3R)$ all cancel. We separate the fourth term into two parts corresponding to the action of $(\varphi+\psi)$ on $\rho$ and on $3\rho^*$.  In the fifth term we separate into the parts corresponding to $3\varphi(\rho-\rho^*)$ and $-\psi(\rho-\rho^*)$ and these contribution cancel each other. Recalling that $\tau-\tau^*=4(m^2+n^2)$  the sixth and seventh terms are easily evaluated.  Thus we have
$$
0=B(Z_1,Z_2,Z_2,Z_1)=\frac{1}{16(m+n+3)}(-16m-16n)+\frac{\tau-3(m^2+n^2)}{(m+n+2)(m+n+3)} \, .
$$Solving for $\tau$ we obtain

\begin{equation}\label{eqn:(4.1)}
\tau=2m(2m+1)+2n(2n+1)+2mn.
\end{equation}

\vskip6pt
\noindent
Stage 1, Step 2:\\
Since metric contact pairs of type (0,0) are not of interest, we suppose that $m>0$.  Then we may choose the unit vector $X$ in 
$T{\FF}_2\cap{\mathcal{H}}$.
Calculating as above using the results of Lemmas \ref{Lemma-2.4} and \ref{Lemma-2.6} we have the following.
\begin{align*}
0=&B(X,Z_1,Z_1,X)=1+\frac{1}{4(m+n)}(-(2m-2n))\\
+&\frac{1}{16(m+n+3)}(-(2m+2n)-2\rho(X,X))+\frac{3}{16(m+n+3)}(-2\rho^*(X,X))\\
+&\frac{3}{16(m+n-1)}(-(2m+2n)-2\rho(X,X)+2\rho^*(X,X))\\ 
+&\frac{\tau-3(m^2+n^2)}{4(m+n+2)(m+n+3)}+\frac{3(m^2+n^2)}{4(m+n-1)(m+n)}\\
=&1-\frac{m-n}{2(m+n)}+{\rm the}\;{\rm additional}\;{\rm terms}. \numberthis \label{eqn:(4.2)}
\end{align*}
Similarly
\begin{align*}
0=&B(X,Z_2,Z_2,X)\\
=&\frac{1}{4(m+n)}(-(2n-2m))+{\rm the}\;{\rm same}\;{\rm additional}\;{\rm terms}.
\end{align*}
Subtracting these two equations we have
$$0=1-\frac{m-n}{m+n}=\frac{2n}{m+n}$$
giving $n=0$ and, from \eqref{eqn:(4.1)}, $\tau=2m(2m+1)$.

\vskip6pt
\noindent
Stage 1, Step 3:\\
 Returning to equation \eqref{eqn:(4.2)} with $n=0$ and using Lemma \ref {Lemma-2.5} we have
\begin{align*}
0=&B(X,Z_1,Z_1,X)=\frac{1}{2}
+\frac{1}{16(m+3)}(-2m-2\rho(X,X))\\
&+\frac{3}{16(m+3)}(-2\rho(X,X)+4m-2)
+\frac{3}{16(m-1)}(-6m+2)\\
&+\frac{\tau-3m^2}{(m+2)(m+3)}+\frac{3m^2}{4m(m-1)}.
\end{align*}
With $\tau=2m(2m+1)$ we may solve for $\rho(X,X)$  and we have
$$\rho(X,X)=2m\; {\rm and}\;\rho^*(X,X)=1.$$

\vskip6pt
\noindent
Stage 1, Step 4:\\
The vector field $Z_2$ restricted to a leaf of $\mathcal{F}_2$ is the normal to the leaf as a submanifold. By Lemma \ref{Lemma-2.2}, $Z_2$ is parallel along the submanifold giving that the leaves are totally geodesic submanifolds. Moreover its own integral curves on $M$ are geodesics.  Therefore $M$ is a local Riemannian product.

Finally with $X$ as above, we compute $0=B(X,\phi X,\phi X,X)$ obtaining
$$0=B(X,\phi X,\phi X,X)=R(X,\phi X,\phi X,X)-1.$$
Therefore the leaves of ${\FF}_2$ are Sasakian manifolds of constant $\phi$-sectional curvature +1.  With $n=0$ and $m+n+1>2$, the leaves are of dimension $\geq 5$; but a Sasakian manifold of constant $\phi$-sectional curvature +1 and dimension $\geq 5$ must be of constant curvature +1 (see e.g. \cite[p. $139$]{Blair}).

As a result the universal cover of $M$ is $\bS^{2m+1}(1)\times{\mathbb R}$, completing Stage 1 of the proof.

\vskip6pt
\noindent
Stage 2:\\
  The proof of the theorem in dimension 4 is very similar but using the formula for the Bochner tensor in this dimension.  Since $m+n+1=2$, one of $m$ or $n$ vanishes and we take $n=0$ at the outset.  The computation of $0=B(Z_1,Z_2,Z_2,Z_1)$ is straightforward and gives, as in the derivation of equation \eqref{eqn:(4.1)},
$$\rho_{11}+\rho_{22}=1+\frac{\tau}{6}.$$
From Lemma \ref{Lemma-2.6}, $\rho_{11}=2$ and $\rho_{22}=0$, therefore $\tau=6$. We also have for $X$ unit and horizontal
$\rho(X,X)=\rho(JX,JX)$ and we then obtain $\rho(X,X)=2$.  Since $X$ was arbitrary as a unit horizontal vector we also have
$\rho(\frac{X+JX}{\sqrt 2},\frac{X+JX}{\sqrt 2})=2$ and hence
$\rho(X,JX)=0$.

As in Stage 1, Step 4, $M$ is  locally the Riemannian product of
a Sasakian manifold $N^3$ of ${\mathbb R}$.  Since $Z_1$ is the Reeb vector field of a Sasakian manifold, $Z_1$ is pointwise an eigenvector of the Ricci operator (see e.g. \cite[p. 113]{Blair}), which gives us $\rho(Z_1,X)=\rho(Z_1,JX)=0$. Therefore, with $\rho(X,JX)=0$ and 
$\rho_{11}=\rho(X,X)=\rho(JX,JX)=2$ we have that $N^3$ is Einstein and in turn of constant curvature +1.

\vskip6pt
\noindent
Stage 3:\\
 Recalling that $J=\phi -\a_2\otimes Z_1+\a_1\otimes Z_2$ and 
$T=\phi +\a_2\otimes Z_1-\a_1\otimes Z_2$,  note that on horizontal vectors these are the same and that they act with opposite signs on vertical vectors.  In particular this means that interchanging the roles of the type numbers $m$ and $n$, we have that the Bochner tensors $B_J$ and $B_T$ are interchanged.  Thus to show that $B_T=0$ implies that $M$ is locally isometric to the Hopf manifold, it is enough to return to Stage 1, Step 2
and work through the rest of the proof starting with $n>0$ instead of $m>0$ and $X\in T{\FF}_1\cap{\mathcal{H}}$.  This proceeds in the same way giving $m=0$, etc.  In the $4$-dimensional case, work with $m=0$ to begin with.
\end{proof}

\begin{remark}
Observe that a priori the two Bochner tensors are not the same.  For example on a, not necessarily Bochner-flat, normal metric contact pair $B_T(X,JX,Z_1,Z_2)$ is the negative of $B_J(X,JX,Z_1,Z_2)$ for a 
horizontal $X$.
\end{remark}

In the Stage $2$ of the proof of Theorem \ref{Theorem-1}, we only use the normality condition and the assumption $B(Z_1, Z_2,Z_2, Z_1)=0$. Hence:
\begin{theorem}
A normal metric contact pair of dimension $4$ on which either $B_J(Z_1, Z_2,Z_2, Z_1)$ or $B_T(Z_1, Z_2,Z_2, Z_1)$ vanishes, is locally isometric to the Hopf manifold $\bS^3 (1) \times \bS^1$.
\end{theorem}

Turning to the conformally flat question, we have the following theorem as a corollary of our Theorem \ref{Theorem-1}.

\begin{theorem}\label{Theorem-2}
Let $(M,\a_1,\a_2,\phi,g)$ be a normal metric contact pair with decomposable $\phi$.  If $M$ is conformally flat, then it is locally isometric to the Hopf manifold $\bS^{2n+1} (1) \times \bS^1$. 
\end{theorem}
\begin{proof}  As a complex manifold, if a normal metric contact pair is conformally flat, it is locally conformal to
${\mathbb C}^{n+1}$ and ${\mathbb C}^{n+1}$ is Bochner-flat.  Since the Bochner tensor is a conformal invariant \cite[Theorem $11.1$]{TV} the manifold is Bochner-flat and hence locally isometric to the Hopf manifold by Theorem \ref{Theorem-1}.
\end{proof}
\vspace{1cm}
We end the paper with an application to Vaisman manifolds. Recall that on a Vaisman manifold the Lee form is parallel, then it has constant length $c$. The manifold is non-K\"ahler if and only if $c\neq 0$ and after a constant rescaling of the metric we can achieve that $c=1$. In \cite{BK} it was shown that, up to constant rescaling of the metric, there is a bijection between non-K\"ahler Vaisman manifolds and normal metric contact pairs of type $(n,0)$ (the latter corresponding to $c=1$).

Therefore we can apply our previous results to give a classification of locally conformally flat and Bochner-flat non-K\"ahler Vaisman manifolds:
\begin{corollary}
Let M be a $(2n+2)$-dimensional non-K\"ahler Vaisman manifold.  If M is either Bochner-flat or locally conformally flat then, after a rescaling of the metric by the inverse of the length of the Lee form, it is locally isometric to the Hopf manifold $\bS^{2n+1} (1) \times \bS^1$.
\end{corollary}
For the locally conformally flat case, the result was already proven by Vaisman \cite[Theorem $3.8$]{V}. The Bochner-flat case was proven by Kashiwada \cite{Ka} which showed an equivalence between locally conformally flatness and Bochner-flatness for Vaisman manifolds.

\end{document}